\documentclass[a4paper,reqno]{amsart}
\date{Last changed 30 September 2020 by PBG}
\usepackage{amssymb, amsmath, amscd}
\usepackage[final]{graphicx}
\usepackage{float}\usepackage{wrapfig}
\usepackage{color}
\usepackage{bbm}
\usepackage[final]{graphicx}

\def\aa{\mathfrak{a}^{\operatorname{Dol}}}
\def\AA{\mathfrak{A}^{\operatorname{Dol}}}
\def\Dol{{\operatorname{Dol}}}\def\pone{\text{\bf 1}}
\def\mm{\mathfrak{m}}\def\dvol{\operatorname{dvol}}
\def\ch{\operatorname{ch}}\def\Td{\operatorname{Td}}
\def\End{\operatorname{End}} 
\def\CC{\mathfrak{C}} 
\def\KK{\mathfrak{K}} 
\def\JJ{\mathfrak{J}} 
\def\RR{\mathfrak{R}} 
\def\SS{\mathfrak{S}} 
\def\TT{\mathfrak{T}} 
\def\UU{\mathfrak{U}} 
\def\VV{\mathfrak{V}} 

\newdimen\theight
\def\TeXref#1{%
             \leavevmode\vadjust{\setbox0=\hbox{{\tt
                     \quad\quad  {\small \textrm #1}}}%
             \theight=\ht0
             \advance\theight by \lineskip
             \kern -\theight \vbox to
             \theight{\rightline{\rlap{\box0}}%
             \vss}%
           }}%

\newtheorem{theorem}{Theorem}[section]

\newtheorem{lemma}[theorem]{Lemma}
\newtheorem{remark}[theorem]{Remark}

\newtheorem{example}[theorem]{Example}
\makeatletter
 \@addtoreset{equation}{section}
\makeatother
\begin{document}
\def\id{\operatorname{id}}\def\dvol{\operatorname{dvol}}
\title[Derived Heat Trace Asymptotics]{Derived Heat Trace Asymptotics for the de Rham and Dolbeault complexes}
\begin{abstract}We examine the derived heat trace asymptotics in both the real and the complex settings for a generalized Witten perturbation. If the dimension
is even, in the real context we show the integral of the local density for the derived heat trace asymptotics is half the Euler characteristic
of the underlying manifold. In the complex context, we assume the underlying geometry is K\"ahler and show the integral of the
local density for the derived heat trace asymptotics defined by the Dolbeault complex is a characteristic number of the complex
tangent bundle and the twisting vector bundle. We identify this characteristic number if the real dimension is $2$ or $4$.
In both the real and complex settings, the local density differs from the corresponding characteristic class by a divergence term.
   \par\noindent \textbf{Keywords:}
   Witten deformation, local index density, de~Rham complex, Dolbeault complex, derived heat trace asymptotics.
\par\noindent\textbf{Subject Classification:} 58J20
\end{abstract}

\author
{J. \'Alvarez L\'opez and P. Gilkey}
\address{PBG: Mathematics Department, University of Oregon, Eugene OR 97403-1222, USA}
\email{gilkey@uoregon.edu}
\address{JAL:  Department of Geometry and Topology, Faculty of Mathematics,
University of Santiago de Compostela, 15782 Santiago de Compostela,
Spain}
\email{jesus.alvarez@usc.es}
\dedicatory{Tribute to Louis Nirenberg}

\maketitle

\section{Introduction}

\subsection{Motivation} Let $h$ be a smooth function on a compact Riemannian manifold $\mathcal{M}=(M,g)$
without boundary of dimension $m$. In this context, Witten~\cite{W82} defined a perturbed de~Rham differential 
$d_h:=d+\operatorname{ext}(dh)$
where $\operatorname{ext}(\cdot)$ denotes left exterior multiplication. Since $d_h=e^{-h}\circ d\circ e^h$, $d_h$ is gauge
equivalent to $d$ and consequently the Betti numbers are unchanged. Let $\operatorname{int}(\cdot)$ denote left
interior multiplication. The perturbed de~Rham co-differential is given by $\delta_h:=\delta+\operatorname{int}(dh)$
 and the perturbed Laplacian is given by $\Delta_h:=d_h\delta_h+\delta_hd_h$. For a generic metric $g$,
 $\Delta_h$ and $\Delta$ will not be gauge equivalent and their spectra will be different. 
 Witten obtained the Morse inequalities by analyzing the spectral asymptotics of the family of operators $\Delta_{sh}$ as $s\to\infty$ when
 $h$ is a Morse function;  we also refer to subsequent work of Bismut and Zhang~\cite{BZ92} and of Helffer and Sj\"ostrand~\cite{HS85}.
 One can replace $dh$ by a real closed 1-form $\omega$ to define
 $d_\omega:=d+\operatorname{ext}(\omega)$. Since $d_\omega$ need not be gauge equivalent to $d$, the twisted Betti numbers
 $\beta^p_\omega$ can be different. However the numbers $\beta^p_{s\omega}$ defined by the family $s\rightarrow s\omega$ for
 $s\in\mathbb{R}$ have well
 defined ground values which are called the Novikov numbers and only depend upon the cohomology
 class defined by $\omega$ in de~Rham cohomology. If $\omega$ is
 of Morse type, these Novikov numbers satisfy a generalized Morse inequality; we refer to work of
 Braverman and Farber~\cite{BF97}, Pazhitnov~\cite{Pa87}, and other authors \cite{BH01,BH08,HM06,Mi15}.

We studied the local index density for Witten deformation of the de Rham and Dolbeault complexes in two previous papers.
In \cite{ALG20}, we used invariance theory to show that the local index density of the twisted de Rham complex
is the Euler form if $m$ is even, and vanishes if $m$ is odd; in particular, it does not depend on $\omega$. A different proof of 
this result was given proviously by the first author, Kordyukov, and Leichtnam~\cite{AKL}, where it was
 applied to study certain trace formulas for foliated flows. Let $(E,h)$ be a Hermitian vector bundle over
 a K\"ahler manifold $(M,g,J)$. In analogy to the real setting, one may introduce the perturbed Dolbeault operator
 $\bar\partial_{\bar\omega}=\bar\partial+\operatorname{ext}(\bar\omega)$ mapping
 $C^\infty(E\otimes\Lambda^{p,q}(M))\rightarrow C^\infty(E\otimes\Lambda^{p,q+1}(M))$ 
 where $\omega$ is a $\partial$ closed form of type $(1,0)$. If $\omega=0$, it follows from the work of 
 Atiyah, Bott, and Patodi~\cite{ABP73} and the work
 the second author~\cite{G73,G73a} that the local index density is 
 $\{\operatorname{Td}(M,g,J)\wedge\operatorname{ch}(E,h)\}_{m}$ as we shall explain presently; 
 this result can fail if the metric is not assumed to be K\"ahler.
 In \cite{ALG20a}, we showed that the index density of $\bar\partial_{\bar\omega}$ exhibits non-trivial dependence on $\omega$;
 we summarize the results of \cite{ALG20,ALG20a} below in Theorem~\ref{T1.1}.
 
G\"unther and Schimming~\cite{GS77} defined a sequence of secondary heat invariants for the de~Rham complex of
which the first is called the derived heat invariant. In a different form, 
this invariant was used by Ray and Singer~\cite{RS71} to study the analytic torsion. In the perturbed setting, define 
$\mathfrak{a}^{\operatorname{deR}}_{m,n}:=\sum_p(-1)^pp\cdot a_{m,n}(x,\Delta_\omega^p)$, where $\Delta_\omega^p$
is the restriction of $\Delta_\omega$ to forms of degree $p$, and $a_{m,n}(x,\Delta_\omega^p)$ is its heat trace invariant of order
$n$. Another version of this invariant, $\mathfrak{a}^{\operatorname{Dol}}_{m,n}$, is similarly defined for the Witten-Novikov
pertubation of the Dolbeault complex in the K\"ahler setting. Our results for these invariants are given below in Theorem~\ref{T1.3}.
We summarize those results as follows.
Consider first the Riemannian setting. Let $\mathcal E_{m,2k}$ denote the Euler invariant of order $2k$ (see Section~\ref{S1.3}). 
We prove that $\mathfrak{a}^{\operatorname{deR}}_{m,n}$ vanishes if $n<m-1$, and exhibits a nontrivial dependence on $\omega$ 
for even $n\ge m$. We show that $\mathfrak{a}^{\operatorname{deR}}_{m,m-1}=\mathcal E_{m,m-1}$ if $m$ is odd, and that
$\int_M\mathfrak{a}^{\operatorname{deR}}_{m,m}\dvol=\frac m2\int_M\mathcal{E}_{m,m}\dvol$ if $m$ is even. In particular,
$\int_M\mathfrak{a}^{\operatorname{deR}}_{m,m}\dvol$ is independent of $\omega$, which is relevant in \cite{AKLa}
to study certain zeta invariants associated to closed $1$-forms (our original motivation).
In the K\"ahler setting, the situation is more complicated. We show that $\mathfrak{a}^{\operatorname{Dol}}_{m,n}=0$ if
$n<m-2$, 
and that $\mathfrak{a}^{\operatorname{Dol}}_{m,n}$ exhibits a nontrivial dependence on $\omega$ for even $n\ge m-2$.
We describe $\mathfrak{a}^{\operatorname{Dol}}_{m,m-2}$ as a perturbation of
$\frac1{(\mm-1)!}g(\operatorname{Td}(M,g,J)\wedge\operatorname{ch}(E,h),\Omega^{\mm-1})$. We prove that
$\int_M\mathfrak{a}^{\operatorname{Dol}}_{m,m}\dvol$ is a characteristic invariant independent of $\omega$.
We determine $\int_M\mathfrak{a}^{\operatorname{Dol}}_{m,m}\dvol$ in general if $m=2$ or $m=4$ 
in terms of characteristic classes of the complex tangent bundle of $M$ and the twisting bundle $E$. This is a global result as the
local invariants $\aa_{m,m}$ exhibit non-trivial dependence on the twisting $(1,0)$ form $\omega$.

\subsection{The real setting}
Let $\mathcal{M}:=(M,g,\omega)$ where $(M,g)$ is an $m$-dimensional Riemannian manifold without boundary which is equipped with
an auxiliary real closed 1-form $\omega$. Let $\operatorname{dvol}$ be the associated Riemannian measure on $M$.
Let $\JJ_{m,n}^p$ be the space of smooth $p$-form valued 
invariants which are homogeneous of weight $n$
in the derivatives of the metric and $\omega$; we refer to Section~\ref{S2} for a precise definition. These spaces
vanish for $p+n$ odd.
The First Theorem of Invariants of Weyl~\cite{W46} shows that such
invariants can be expressed in terms of contractions of indices in pairs and alternations of $p$ indices where the indices
range from $1$ to $m$. Let $R_{ijkl}$ be the components of the curvature tensor and let $\omega_i$ be the components of $\omega$.
The scalar curvature $\tau\in\JJ_{m,2}^0$, 
and norm squared $\|\omega\|^2\in\JJ_{m,2}^0$ take the form
$$
\tau=\sum_{i,j,k,\ell=1}^mg^{i\ell}g^{jk}R_{ijkl}\quad\text{and}\quad\|\omega\|^2=\sum_{j,k=1}^mg^{jk}\omega_j\omega_k\,.
$$
The restriction map 
$r:\JJ_{m,n}^p\rightarrow\JJ_{m-1,n}^p$
 is defined by restricting the range of summation to range from 1 to $m-1$ rather than from 1 to $m$. 
 We shall discuss the restriction map further in Section~\ref{S2}.
 
 \subsection{The Euler form}\label{S1.3}
Let $\mathcal{E}_{m,2k}$ be the integrand of the Chern-Gauss-Bonnet
Theorem~\cite{C44}:
\begin{eqnarray*}
&&\mathcal{E}_{m,2k}:=\sum_{i_1,\dots,i_k,j_1,\dots,j_k=1}^m\frac{(-1)^{k}}{8^{k}\pi^{k}k!}g(e^{i_1}\wedge\dots\wedge e^{i_k},e^{j_1}\wedge\dots\wedge e^{j_k})\\
&&\qquad\qquad\qquad\qquad\qquad\qquad\qquad R_{i_1i_2j_1j_2}\dots R_{i_{k-1}i_kj_{k-1}j_k}\in\JJ_{m,2k}^0\,.
\end{eqnarray*}
This element is universal; $r(\mathcal{E}_{m,2k})=\mathcal{E}_{m-1,2k}$ for any $m$. One has, for example,
$$
\mathcal{E}_{m,2}:=(4\pi)^{-1}\tau\text{ and }\mathcal{E}_{m,4}:=(32\pi^2)^{-1}\{\tau^2-4\|\rho\|^2+\|R\|^2\}
$$
where $\rho$ is the Ricci tensor. Similar formulas for $\mathcal{E}_{m,6}$ and $\mathcal{E}_{m,8}$ can be found in
Pekonen~\cite{P88}. Let $\delta$ be the co-derivative
$$
\delta:\JJ_{m,n-1}^{1}\rightarrow\JJ_{m,n}^0\,.
$$

\subsection{The K\"ahler setting} Let $\mathcal{K}:=(M,g,J,E,h,\omega)$ where $(M,g,J)$ is a K\"ahler manifold
of real dimension $m=2\mm$,
$(E,h)$ is a holomorphic vector bundle over $M$ which is equipped with a Hermitian fiber metric $h$, and $\omega$ is a $\partial$ closed
$(1,0)$ form. Let $\Omega$ be the K\"ahler form of $\mathcal{K}$; $\operatorname{dvol}=\frac1{\mm!}\Omega^\mm$. 
Let $\KK_{m,n}^{k}$ 
be the corresponding space of invariants in the complex setting; we refer to Section~\ref{S2}
for precise details. These spaces vanish for $k+n$ odd. Again, we have a natural restriction map
$r:\KK_{m,n}^{k}\rightarrow\KK_{m-2,n}^{k}$.
We can pair form valued invariants with an appropriate power of the K\"ahler form to obtain scalar invariants;
if $P\in\KK_{m,2k}^{2k}$ then $g(P,\Omega^k)\in\KK_{m,2k}^0$. For example, if $k=\mm$, then
the Hodge $\star$ operator takes the form
$\star P=\frac1{\mm!}g(P,\Omega^\mm)$.

\subsection{The ring of characteristic forms} Let $c_k$, $\ch_k$, and $\Td_k$ be the $k^{\operatorname{th}}$ Chern class,
Chern character, and Todd class, respectively (see~\cite{H66}).
For example,
$$\begin{array}{llll}
\ch_0=\dim(E),&\ch_1=c_1,&\ch_2=\frac12(c_1^2-2c_2),&\ch_3=\frac16(c_1^3-3c_1c_2+3c_3),\\
\textstyle\Td_0=1,&\Td_1={\frac {c_{1}}{2}},&
\quad\Td_2={\frac {c_{1}^{2}+c_{2}}{12}},&
\Td_3={\frac {c_{1}c_{2}}{24}}\,.
\end{array}$$
Let $T_cM:=(TM,J)$ be the associated complex tangent bundle. We decompose the graded ring of characteristic forms
$$
\CC_m:=\mathbb{C}[\ch_1(T_cM),\dots,\ch_\mm(T_cM),\ch_1(E,h),\dots,\ch_\mm(E,h)]
$$
into homogeneous components
$\CC_m=\oplus_k\CC_m^{2k}$ where
$\CC_m^{2k}\subset\KK_{m,2k}^{2k}$.

\subsection{The Witten deformation} In the real setting, let $d_\omega:=d+\operatorname{ext}(\omega)$ be
the Witten deformation of the exterior derivative; the adjoint is then $\delta_\omega:=\delta+\operatorname{int}(\omega)$. 
We may decompose $\Delta_{\mathcal{M}}=\oplus_p\Delta_{\mathcal{M}}^p$
where $\Delta_{\mathcal{M}}:=d_\omega\delta_\omega+\delta_\omega d_\omega$
is the associated Laplacian and where
$\Delta_{\mathcal{M}}^p$ is a self-adjoint elliptic operator of Laplace type on $C^\infty(\Lambda^p(M))$. Similarly,
in the complex setting, let $\bar\partial_\omega=\bar\partial+\operatorname{ext}(\bar\omega)$ define the
deformation of the Dolbeault operator; the adjoint is 
then $\delta^{\prime\prime}_\omega=\delta^{\prime\prime}+\operatorname{int}(\omega)$.
Decompose
$\Delta_{\mathcal{K}}=\oplus_{p,q}\Delta_{\mathcal{K}}^{p,q}$ where $\Delta_{\mathcal{K}}:=2(\bar\partial_{\bar\omega}\delta^{\prime\prime}_\omega+\delta^{\prime\prime}_\omega\bar\partial_{\bar\omega})$ is the associated Laplacian and where
$\Delta_{\mathcal{K}}^{p,q}$ is a self-adjoint operator
of Laplace type on $C^\infty(\Lambda^{p,q}(M)\otimes E)$.

\subsection{The local index density}
If $\Delta$ is an operator of Laplace type on a compact Riemannian manifold of dimension $m$,
let $a_{m,n}(x,\Delta)$ be the local heat trace asymptotics. If $f\in C^\infty(M)$, then
$$
\operatorname{Tr}_{L^2}\{fe^{-t\Delta}\}\sim\sum_{n=0}^\infty\int_Mf(x)a_{m,n}(x,\Delta)\dvol\quad\text{as}\quad t\downarrow0\,.
$$
The invariants $a_{m,n}(x,\Delta)$ vanish for $n$ odd.
We introduce the local index densities for the de~Rham complex and Dolbeault complex by setting:
\begin{eqnarray*}
&&\displaystyle a_{m,n}^{\operatorname{deR}}(x):=\sum_p(-1)^pa_{m,n}(x,\Delta_{\mathcal{M}}^p)
\in\JJ_{m,n}^0,\\
&&\displaystyle a_{m,n}^{\operatorname{Dol}}(x):=\sum_p(-1)^pa_{m,n}(x,\Delta_{\mathcal{K}}^{0,p})
\in\KK_{m,n}^0\,.
\end{eqnarray*}
A cancellation argument due to Bott yields:
\begin{equation}\label{E1.a}\begin{array}{l}
\displaystyle\int_Ma_{m,n}^{\operatorname{deR}}(x)\dvol(g)=\left\{\begin{array}{ll}0&\text{ for }n\ne m\\
\text{Euler characteristic}(M)&\text{ for }n=m\end{array}\right\},\\[.15in]
\displaystyle\int_Ma_{m,n}^{\operatorname{Dol}}(x)\dvol(g)=\left\{\begin{array}{ll}0&\text{ for } n\ne m\\
\text{arithmetic genus}(M,E)\hspace{.1cm}&\text{ for }n=m\end{array}\right\}\,.
\end{array}\end{equation}
Let $\Im(\omega)$ be the imaginary part of $\omega$.
Set $\Theta:=\displaystyle\sum_k\frac{1}{k!\pi^k}\{d\Im(\omega)\}^k$.
We have the following previous results~\cite{ALG20,ALG20a} of the authors.
\begin{theorem}\label{T1.1}
$a_{m,n}^{\operatorname{deR}}=\left\{\begin{array}{cll}0&\hspace{2.9cm}&\text{\rm if }n<m\\
\mathcal{E}_{m,m}&&\text{\rm if }n=m\end{array}\right\}$,
\medbreak\hglue 2.2cm$a_{m,n}^{\operatorname{Dol}}=\left\{\begin{array}{ll}0&\text{\rm if }n<m\\
\frac1{\mm!}g(\Omega^{\mm},\left\{\operatorname{Td}(T_cM)\wedge\operatorname{ch}(E)\wedge\Theta\right\}_{m})
&\text{\rm if }n=m\end{array}\right\}$. 
\end{theorem}
\begin{remark}\rm We note that this result for $a_{m,n}^{\operatorname{Dol}}$
can fail in the Hermitian setting. Although Equation~(\ref{E1.a}) continues to hold if $(M,g,J)$
is only assumed Hermitian, it is necessary to restrict to the context of K\"ahler geometry to identify the local index density
with a characteristic form even if $E$ is trivial and $\omega=0$. We refer to Gilkey,  Nik\v cevi\'c,\ and Pohjanpelto~\cite{GNP97} for details.
\end{remark}

\subsection{Derived invariants}
The following is the first in a sequence of invariants introduced by G\"unther and Schimming~\cite{GS77} in the real category
(see the discussion on page 181 of Gilkey~\cite{G95}); 
a related invariant appears in the discussion by Ray and Singer~\cite{RS71} of analytic torsion. 
Define
\begin{eqnarray*}
&&\mathfrak{a}^{\operatorname{deR}}_{m,n}:=\sum_p(-1)^pp\cdot a_{m,n}(x,\Delta^p_{\mathcal{M}})
\in\JJ_{m,n}^0\,,\\
&&\mathfrak{a}^{\operatorname{Dol}}_{m,n}:=\sum_p(-1)^pp\cdot a_{m,n}(x,\Delta^{0,p}_{\mathcal{K}})
\in\KK_{m,n}^0\,.
\end{eqnarray*}

\goodbreak\begin{theorem}\label{T1.3}\rm
\ \begin{enumerate}
\item Let $\mathcal{M}=(M,g,\omega)$ be a Riemannian manifold which is equipped with an auxiliary closed 1-form $\omega$.
\begin{enumerate}
\item If $n<m-1$, then $\mathfrak{a}^{\operatorname{deR}}_{m,n}=0$.
\item If $m$ is odd, then $\mathfrak{a}^{\operatorname{deR}}_{m,m-1}=\mathcal{E}_{m,m-1}$ is independent of $\omega$.
\item If $m$ is even, then $\mathfrak{a}^{\operatorname{deR}}_{m,m}=\frac m2\mathcal{E}_{m,m}+\delta Q_{m,m-1}^1$ for 
$Q_{m,n-1}^1\in\JJ_{m,n-1}^{1}$.
\end{enumerate}
\item Let $\mathcal{K}=(M,g,J,E,h,\omega)$ where $(M,g,J)$ is a K\"ahler manifold, $(E,h)$ is a holomorphic Hermitian
vector bundle over $M$, and $\omega$ is an auxiliary $\partial$ closed form of type $(1,0)$.
\begin{enumerate}
\item If $n<m-2$, $\mathfrak{a}^{\operatorname{Dol}}_{m,n}(\mathcal{K})=0$.
\item If $n=m-2$, 
{$\mathfrak{a}^{\operatorname{Dol}}_{m,m-2}=\frac1{(\mm-1)!}
g(\operatorname{Td}(M,g,J)\wedge\operatorname{ch}(E,h)\wedge\Theta,\Omega^{\mm-1})$.}
\par\noindent There exists $Q_{m,m-3}^1\in\KK_{m,m-3}^1$ so
\par\noindent\quad
$\mathfrak{a}^{\operatorname{Dol}}_{m,m-2}=\frac1{(\mm-1)!}
g(\operatorname{Td}(M,g,J)\wedge\operatorname{ch}(E,h),\Omega^{\mm-1})+\delta Q_{m,m-3}^1$.
\item We have $
\mathfrak{a}^{\operatorname{Dol}}_{m,m}=\frac1{\mm!}g(R_{m}^m,\Omega^{\mm})+\delta Q_{m,m-1}^1$ for 
 $Q_{m,m-1}\in\KK_{m,m-1}^1$ and $R_{m}^m\in\CC_{m}^m$.
 \end{enumerate}\end{enumerate}
\end{theorem}

Assertion~(1c) and Assertion~(2c) show that 
$\int_M\mathfrak{a}_{m,m}^{\operatorname{deR}}\dvol$ is a characteristic number which is
 independent of $\omega$ in the Riemannian setting and $\int_M\mathfrak{a}_{m,m}^{\operatorname{Dol}}\dvol$ is a characteristic number which is
 independent of the structures in the K\"ahler setting. In Theorem~\ref{T1.3}~(2b), we may eliminate $\Theta$ at the cost of
introducing an additional divergence term. The divergence terms are in general present. We have, for example, $Q_{2,1}^1=(2\pi)^{-1}\Theta$ in Assertion~(2b).
We will establish the following result in Section~\ref{S7}. It shows, in particular, that $\aa_{m,m}$ is not a
multiple of $a_{m,m}^{\Dol}$.

\begin{theorem}\label{T1.4}
Use Theorem~\ref{T1.3} to express $\aa_{m,m}=\frac1{\mm!}g(R_m^m,\Omega^\mm)+\delta Q_{m,m-1}^1$
for $Q_{m,m-1}^1\in\KK_{m,m-1}^1$ and  $R_{m}^m\in\CC_{m}^m$.
\begin{enumerate}
\item Let $\mathcal{M}=(M,g,J,E,h)$ where $(M,g,J)$ is a K\"ahler manifold and $(E,h)$ is a Hermitian vector bundle bundle.
\begin{enumerate}
\item If $m=2$, $R_2^2
=\frac13c_1(T_cM)\ch_0(E)+\frac12c_1(E)$.
\item If $m=4$, $R_4^4=(\Td_2+\frac1{24}c_1^2)(T_cM)\ch_0(E)
+\frac7{12}c_1(T_cM)c_1(E)+\ch_2(E)$.
\end{enumerate}
\item Let  $\mathcal{M}=\mathcal{M}_1\times\dots\times\mathcal{M}_\mm$
where $\mathcal{M}_i=(M_i,g_i,J_i)$ are Riemann surfaces and $(E_i,h_i)$ are Hermitian line bundles.
\ \begin{enumerate}
\item If $(M_i,g_i,J_i)$ are flat tori,
$R_m^m(\mathcal{M})=
\frac12\mm a_{m,m}^{\Dol}(\mathcal{M})=\frac12\mm\ch_\mm(E)$.
\item If $(E_i,h_i)$ is trivial for all $i$,
$R_m^m(\mathcal{M})=\frac{2}3\mm a_{m,m}^{\Dol}(\mathcal{M})=\frac{2}3\mm\Td_\mm(T_cM)$.
\end{enumerate}
\end{enumerate}
\end{theorem}

\section{Spaces of invariants}\label{S2}

\subsection{Spaces of local formula in the real setting} It is necessary to be a bit careful concerning
what we mean by a local formula. Let $x=(x^1,\dots,x^m)$ be local coordinates centered at a point $P$ of
$M$. Let $U=(i_1,\dots,i_a)$ be a collection of indices. Decompose $\omega=\omega_idx^i$. Set $|U|:=a$,
$\partial_i:=\frac{\partial}{\partial x^i}$,
$$\begin{array}{lll}
\partial_x^U:=\partial_{i_1}\dots\partial_{i_a},&
g_{ij}:=g(\partial_i,\partial_j),&
g_{ij/U}:=\partial_x^Ug_{ij},\\[0.05in]
\omega_{i/U}:=\partial_x^U\omega_i,&
\operatorname{weight}\{g_{ij/U}\}=|U|,&
\operatorname{weight}\{\omega_{i/U}\}=1+|U|.
\end{array}$$
Let $M_m(\mathbb{R})$ be the ring of $m\times m$ real matrices.
Let $\mathcal{O}_m\subset M_m(\mathbb{R})$ be the open subset of matrices $g$ so that
$g$ defines a positive definite inner product. Let 
$$
\RR:=C^\infty(\mathcal{O}_m)[g_{ij/U},\omega_i,\omega_{i/U}]_{|U|>0}\,.
$$
These are the local formulas in the derivatives of the metric and of $\omega$ with coefficients which
are smooth functions of the Riemannian metric $g$ which we will be considering. The weight induces
a natural grading on $\RR$ and we may decompose $\RR_m=\oplus_n\RR_{m,n}$ into the polynomials
which are homogeneous of weight $n$. If $P\in\RR_m$, we can evaluate $P$ in a coordinate system in the
obvious fashion; we say $P$ is {\it invariant} if the value of $P$ is independent of the particular local
coordinate system chosen and we let $\JJ_m\subset\RR_m$ be the ring of invariant local formulae; we may decompose
$\JJ_m=\oplus_n\JJ_{m,n}$ where $\JJ_{m,n}\subset\RR_{m,n}$. The space of $p$-form valued invariants
$\JJ_{m,n}^p$ is defined similarly.

Clearly we can express the curvature tensor $R$ and its covariant derivatives in terms of the derivatives of the
metric. Conversely, in geodesic coordinates, we can express the derivatives of the metric in terms of the
covariant derivatives of the curvature tensor. The {\it weight} of the curvature tensor $R$ is $2$ since
it is linear in the second derivatives of the metric and quadratic in the first derivatives of the metric.
We increase the weight by 1 for every explicit covariant derivative which
appears. Thus, for example, the scalar curvature $\tau$ and
$\|\omega\|^2$ have weight 2 while the square of the norms $\|\rho\|^2$ and $\|R\|^2$ of the Ricci tensor
and full curvature tensor, respectively,
have weight 4. Similarly, $\|\nabla R\|^2$ has weight 6 and $\mathcal{E}_{m,2k}$ has weight $2k$.

\subsection{Spaces of local formulae in the complex setting}\label{S2.2}
The situation is considerably more delicate here and we must proceed with some care. 
Let $(M,g,J)$ be a Hermitian manifold; we assume $J^*g=g$ but do not impose the K\"ahler condition.
Let $(E,h)$ a holomorphic vector bundle over $M$ of dimension $\ell$.
Fix a point $P$ of $M$. Choose local holomorphic coordinates
$\vec z=(z^1,\dots,z^{\mm})$ centered at $P$ and a local holomorphic frame 
$\vec s=(s_1,\dots,s_\ell)$ for $E$. Let $U=(\alpha_1,\dots,\alpha_a)$
and $V=(\beta_1,\dots,\beta_b)$ be collections of indices. 
Expand $\omega=\omega_\alpha dz^\alpha$ and $\bar\omega=\bar\omega_{\bar\beta}d\bar z^{\bar\beta}$.
Set
$$\begin{array}{llll}
|U|:=a,&|V|:=b,&
\partial_\alpha:=\frac{\partial}{\partial z^\alpha},&
\partial_{z;U}=\partial_{\alpha_1}\dots\partial_{\alpha_a},\\[0.05in]
\partial_{\bar\beta}:=\frac{\partial}{\partial\bar z^{\bar\beta}},&\partial_{\bar z;\bar V}
=\partial_{\bar\beta_1}\dots\partial_{\bar\beta_b},&
g_{\alpha\bar\beta}:=g(\partial_\alpha,\partial_{\bar\beta}),&
h_{p\bar q}:=h(s_p,s_q)\,.
\end{array}$$
Introduce the following notation for the derivatives of the structures involved:
\begin{equation}\label{E2.a}\begin{array}{ll}
g_{\alpha\bar\beta/U\bar V}:=\partial_{z;U}\partial_{\bar z;\bar V}g(\partial_\alpha,\partial_{\bar\beta}),&
h_{p\bar q/U\bar V}:=\partial_{z;U}\partial_{\bar z;\bar V}h(s_p,s_q),\\[0.05in]
\omega_{\alpha/U\bar V}:=\partial_{z;U}\partial_{\bar z;\bar V}\omega_\alpha,&
\bar\omega_{\bar\beta/U\bar V}:=\partial_{z;U}\partial_{\bar z;\bar V}\bar\omega_{\bar\beta}.
\end{array}\end{equation}
If $|U|=0$, there are no holomorphic derivatives, if $|V|=0$, there are no anti-holomorphic derivatives, and if
$|U|+|V|=0$, there are no derivatives at all.  We have
$$
\bar\partial\omega=-\omega_{\alpha/\bar\beta}dz^\alpha\wedge d\bar z^{\bar\beta}\text{ and }
\partial\bar\omega=\bar\omega_{\bar\beta/\alpha}dz^\alpha\wedge d\bar z^{\bar\beta}\,.
$$
Consequently, the variables $\{\omega_\alpha,\omega_{\alpha/\bar\beta},\bar\omega_{\bar\beta},
\bar\omega_{\bar\beta/\alpha}\}$ are tensorial
unlike the remainder of the variables defined in Equation~(\ref{E2.a}).  Define
$$\begin{array}{ll}
\operatorname{weight}\{g_{\alpha\bar\beta/U\bar V}\}=
\operatorname{weight}\{h_{p\bar q/U\bar V}\}=|U|+|V|,\\[0.05in]
\operatorname{weight}\{\omega_{\alpha;U\bar V}\}=
\operatorname{weight}\{\bar\omega_{\bar\beta;U\bar V}\}=1+|U|+|V|.
\end{array}$$
Let $M_k(\mathbb{C})$ be the ring of $k\times k$ complex matrices.
Let $\mathcal{U}\subset M_\mm(\mathbb{C})\otimes M_\ell(\mathbb{C})$ be the open subset consisting of all 
the matrices $(g,h)$
so that $g$ and $h$ define positive definite Hermitian inner products. We consider the polynomial ring
$$
\mathcal{P}_m:=C^\infty(\mathcal{U})[g_{\alpha\bar\beta/U\bar V},h_{p\bar q/U\bar V},
\omega_{\alpha/U\bar V},\bar\omega_{\bar\beta/U\bar V},
\omega_\alpha,\bar\omega_{\bar\beta}]_{|U|+|V|>0}\,.
$$
As in the real setting, we say that $P\in\mathcal{P}_m$ is {\it invariant}
if the evaluation is independent of the particular
coordinate system $\vec z$ and frame $\vec s$ chosen. Let $\KK_m$ be the ring
of invariants in the K\"ahler context; $\KK_m$ is a graded ring and we may
use the weight to decompose $\KK_m=\oplus_n\KK_{m,n}$
into the polynomials which are
homogeneous of weight $n$. For example, $a_{m,n}(\Delta^{p,q})\in\KK_{m,n}$.
The spaces $\KK_{m,n}^p$ of $p$-form valued invariants are defined similarly. For
example, $d\omega\in\KK_{m,2}^2$.

\subsection{Homotheties}
The weight of a polynomial describes its behaviour under homotheties. 
The following is immediate from the definition and
could be used to give an equivalent definition of the weight.
\begin{lemma}\label{L2.1}\rm\ 
\begin{enumerate}
\item Let $P\in\JJ_{m,k}^p$. Then $P(x,c^2g,\omega)=c^{p-k}P(x,g,\omega)$.
\item Let $P\in\KK_{m,k}^p$. Then $P(x,c^2g,J,E,h,\omega)=c^{p-k}P(x,g,J,E,h,\omega)$.
\end{enumerate}
\end{lemma}

\begin{example}\rm It is worth giving some examples to illustrate Lemma~\ref{L2.1}. We work in the real context.
Let $\|\omega\|_g^2:=g^{ij}\omega_i\omega_j$. Since $\omega_i$
has weight 1, $\|\omega\|_g^2$ has weight 2 so $\|\omega\|_g^2\in\JJ_{m,2}^0$. 
If we rescale the metric and set $g_{c;ij}=c^2g_{ij}$, then $g^{ij}_c=c^{-2}g^{ij}$
so $\|\omega\|_{g_c}^2=c^{-2}\|\omega\|_g$. The components of the curvature tensor $R_{ijk}{}^\ell$ have weight 2
since they are linear in the 2-jets of the metric and quadratic in the 1-jets of the metric. The Levi-Civita connection is
unchanged if we rescale the metric and thus $R_{g_c;ijk}{}^\ell=R_{g;ijk}{}^\ell$. Let $\tau_g\in\JJ_{m,2}^0$ be the scalar curvature. We compute that
$\tau_{g_c}=g_c^{jk}R_{ijk}{}^i=c^{-2}g^{jk}R_{ijk}{}^i=c^{-2}\tau_g$. We have $d\tau_g\in\JJ_{m,3}^1$ and
$d\tau_{g_c}=c^{-2}d\tau_g$.
\end{example} 
\subsection{The restriction map}
As noted in the introduction, a spanning set for the space of invariants in the real setting
 is given by contraction of indices in pairs and alternations of
indices where the indices range from $1$ to $m$; there is an analogous result in the complex setting.
If $P\in\JJ_{m,n}^p$, $r(P)\in\JJ_{m-1,n}^p$ is defined by restricting the range of summation
to range from $1$ to $m-1$. Although Weyl's First Theorem of Invariants yields a spanning set, it is not a basis as one has the
Bianchi identities and higher order analogues.
Thus it is not immediately obvious that the restriction map is independent
of the particular expression of an invariant in terms of a Weyl spanning set. 
To get around this difficulty, it is convenient to use a slightly different more geometric formalism.
Let $\mathbb{T}^1=(S^1,d\theta^2,0)$ be the circle with the flat structures. If $\mathcal{N}$ is an $m-1$ dimensional
structure, one forms the $m$ dimensional structure
$$
\mathcal{M}=\mathcal{N}\times\mathbb{T}^1:=(N\times S^1,ds^2_N+d\theta^2,\pi_1^*\omega_N)\,.
$$
Let $i_\theta(x)=(x,\theta)$ be an inclusion of $N$ in $N\times S^1$; 
the particular basepoint $\theta$ chosen is irrelevant since $\mathbb{T}^1$ is homogeneous.
Then $r(P)\in\JJ_{m-1,n}^p$ is characterized by the identity
\begin{equation}\label{E2.b}
r(P)(\mathcal{N}):=i_\theta^*P(\mathcal{N}\times\mathbb{T}^1)\,.
\end{equation}
In the complex setting, the invariance theory is that of the unitary group
rather than the orthogonal group and one sums over pairs of holomorphic and anti-holomorphic indices where the indices
range from 1 to $\mm$. 
Instead of considering $\mathbb{T}^1$ one considers the flat 2-torus $\mathbb{T}^2$ but the remainder of the analysis
is the same and one obtains:

\begin{lemma}\label{L2.3}
\ \begin{enumerate}
\item $r$ is a well defined map from  $\JJ_{m,n}^p$ onto $\JJ_{m-1,n}^p$. 
\item $r$ is a well defined map from  $\KK_{m,n}^p$ onto $\KK_{m-2,n}^p$. 
\end{enumerate}
\end{lemma}

We introduce the ring
$$
\TT_m:=\mathbb{C}[\ch_k(TM,J,g),\ch_k(E,h),d\omega,d\bar\omega,\omega,\bar\omega]\,.
$$
We may decompose $\TT_m:=\oplus_k\TT_m^k$ where $\TT_m^k\subset\KK_{m,k}^{k}$. 
We refer to \cite{ALG20} for the proof of Assertions~(1) and to \cite{ALG20a} for the proof of Assertions~2 in 
the following result; it is necessary to restrict to the K\"ahler setting.
\begin{lemma}\label{L2.4}
\ \begin{enumerate}
\item In the Riemannian setting, we have that:
\begin{enumerate}
\item $r:\JJ_{m,n}^0\rightarrow\JJ_{m-1,n}^0$ is injective if $n<m$.
\item If $m$ is even, then $\ker\{r:\JJ_{m,m}^0\rightarrow\JJ_{m-1,m}^0\}=\mathcal{E}_{m,m}\cdot\mathbb{R}$.
\end{enumerate}
\item In the K\"ahler setting, we have that:
\begin{enumerate}
\item $r:\KK_{m,n}^0\rightarrow\KK_{m-2,n}^0$ is injective if $n<m$.
\item $\ker\{r:\KK_{m,m}^0\rightarrow\KK_{m-2,m}^0\}=\star\TT_m^m$.
\end{enumerate}\end{enumerate}\end{lemma}

We have the following useful result.

\begin{lemma}\label{L2.5}\rm\ 
$r(\mathfrak{a}^{\operatorname{deR}}_{m,n})=-(4\pi)^{-1/2}a_{m-1,n}^{\operatorname{deR}}$ and
$r(\mathfrak{a}^{\operatorname{Dol}}_{m,n})=-(4\pi)^{-1}a_{m-1,n}^{\operatorname{Dol}}$.
\end{lemma}

\begin{proof} Although the lemma in the real setting follows from work of G\"unther and Schimming~\cite{GS77}, we
shall give a direct
proof in the interests of completeness. Let $\mathcal{M}_m=\mathcal{N}_{m-1}\times \mathbb{T}^1$.
We compute:
\medbreak\quad $\Lambda^pM=\Lambda^pN\oplus\Lambda^{p-1}N\wedge d\theta$,
\medbreak\quad $\Delta^p_{\mathcal{M}}=\{\Delta^p_{\mathcal{N}}\otimes\id+\id\otimes\Delta^0_{\mathbb{T}^1}\}
\oplus\{\Delta^{p-1}_{\mathcal{N}}\otimes\id+\id\otimes\Delta^1_{\mathbb{T}^1}\}$,
\medbreak\quad $a_{m,n}((x,\theta),\Delta^p_{\mathcal{M}})=(4\pi)^{-1/2}\{a_{m-1,n}(x,\Delta^p_{\mathcal{N}})
+a_{m-1,n}(x,\Delta^{p-1}_{\mathcal{N}})\}$,
\medbreak\quad $\displaystyle\sum_p(-1)^pp\cdot a_{m,n}((x,\theta),\Delta^p_{\mathcal{M}})$
\medbreak\qquad$=\displaystyle
(4\pi)^{-1/2}\sum_p(-1)^pp\cdot\left\{a_{m-1,n}(x,\Delta^p_{\mathcal{N}})
+a_{m-1,n}(x,\Delta^{p-1}_{\mathcal{N}})\right\}$
\medbreak\qquad$=\displaystyle
(4\pi)^{-1/2}\sum_p(-1)^pa_{m-1,n}(x,\Delta^p_{\mathcal{N}})\{p-(p+1)\}$.
\medbreak\noindent Assertion~1 now follows from Equation~(\ref{E2.b}); the argument is the same for the Dolbeault complex.
\end{proof}

\subsection{Low dimensional computations} We refer to \cite{ALG20} for the following result:
\begin{lemma}\label{L2.6}
$\displaystyle a_{1,2}^{\operatorname{deR}}=-\frac{\delta\omega}{\sqrt\pi}$ and
$\displaystyle a_{2,2}^{\operatorname{Dol}}=\frac{\tau}{8\pi}-\frac{1}{\pi}\delta(\Re(\omega))$.
\end{lemma}

\section{The proof of Theorem~\ref{T1.3}~(1a,2a)}
Let $n<m-1$. By Lemma~\ref{L2.5}, 
$r(\mathfrak{a}^{\operatorname{deR}}_{m,n})=a_{m-1,n}^{\operatorname{deR}}$. 
By Theorem~\ref{T1.1}, $a_{m-1,n}^{\operatorname{deR}}=0$. 
By Lemma~\ref{L2.4}, $r:\JJ_{m,n}\rightarrow\JJ_{m-1,n}$
is injective for $n<m$. This shows that $\mathfrak{a}^{\operatorname{deR}}_{m,n}=0$ which
establishes Theorem~\ref{T1.3}~(1a);
the proof of Theorem~\ref{T1.3}~(2a) is the same in the K\"ahler setting.~\qed

\section{The proof of Theorem~\ref{T1.3}~(1b,2b)}
By Theorem~\ref{T1.1}, $a_{2k,2k}^{\operatorname{deR}}=\mathcal{E}_{2k,2k}$. 
By Lemma~\ref{L2.5}, $r(\mathfrak{a}^{\operatorname{deR}}_{2k+1,2k})=-(4\pi)^{-1/2}a^{\operatorname{deR}}_{2k,2k}$.
By Lemma~\ref{L2.4}, $r:\JJ_{2k+1,2k}\rightarrow\JJ_{2k,2k}$ is injective.
By construction, we have that $r\mathcal{E}_{2k+1,2k}=\mathcal{E}_{2k,2k}$.
Assertion~(1b) follows.
The same argument in the K\"ahler setting shows
$$\mathfrak{a}^{\operatorname{Dol}}_{m,m-2}=\frac1{(\mm-1)!}
g(\operatorname{Td}(M,g,J)\wedge\operatorname{ch}(E,h)\wedge\Theta,\Omega^{\mm-1})
$$
which establishes the first part of Assertion~2b.
It is immediate from the definition that $\Theta=1+d\Phi$. Since $\operatorname{Td}$ and $\operatorname{ch}$ are closed,
\begin{eqnarray*}
&&\{\operatorname{Td}(M,g,J)\wedge\operatorname{ch}(E,h)\wedge\Theta\}_{m-2}\\
&=&\{\operatorname{Td}(M,g,J)\wedge\operatorname{ch}(E,h)\}_{m-2}
+d\{\operatorname{Td}(M,g,J)\wedge\operatorname{ch}(E,h)\wedge\Phi\}_{m-3}\,.
\end{eqnarray*}
Taking the inner product with $\frac1{(\mm-1)!}\Omega^{\mm-1}$ is just the Hodge~$\star$ operator in complex dimension $\mm-1$. Thus
\begin{equation}\label{E4.a}\begin{array}{l}
\frac1{(\mm-1)!}g(\operatorname{Td}(M,g,J)\wedge\operatorname{ch}(E,h)\wedge\Theta,\Omega^{\mm-1})\\[0.05in]
=\frac1{(\mm-1)!}g(\operatorname{Td}(M,g,J)\wedge\operatorname{ch}(E,h),\Omega^{\mm-1})+\delta Q_{m-2,m-3}^1
\end{array}\end{equation}
for $Q_{m-2,m-3}^1=\star \{\operatorname{Td}(M,g,J)\wedge\operatorname{ch}(E,h)\wedge\Phi\}_{m-3}$. We have $\delta r=r\delta$. 
By Lemma~\ref{L2.3} and Lemma~\ref{L2.4}, we can lift Equation~(\ref{E4.a}) to a corresponding equation in 
complex dimension $\mm$ and complete the proof of Assertion~(2b).~\qed

\begin{remark}\rm This gives an explicit description of the 1-form $Q_{m,m-3}^1$ of Theorem~\ref{T1.3}~(2b)
as the unique element $Q_{m,m-3}^1\in\KK_{m,m-3}^1$ such that
\smallbreak\centerline{$r(Q_{m,m-3}^1)=\star\{\operatorname{Td}\wedge\operatorname{ch}\wedge\Phi\}_{m-3}\in\KK_{m-2,m-3}^1$.}
\end{remark}

\section{The proof of Theorem~\ref{T1.3}~(1c)}\label{S5}
\medbreak We continue our discussion by generalizing a result of Gilkey~\cite{G95} 
(see Assertion~4 of Lemma 2.9.1 on page 210) from the context of purely
metric invariants to invariants which also depend upon the derivatives of $\omega$.

\begin{lemma}\label{L5.1}\rm
Let $P\in\JJ_{m,n}$ where $n\ne m$. Suppose $\int_MP(x,g,\omega)\dvol(g)$
is independent of $(g,\omega)$ for all closed $m$-dimensional manifolds $M$. 
Then there exists $Q_{m,n-1}\in\JJ_{m,n-1}^{1}$ so that  $P=\delta Q_{m,n-1}$.
\end{lemma}

\begin{proof}
 Let $f\in C^\infty(M)$ and $\varepsilon\in\mathbb{R}$. We consider the conformal variation
$$
P_1(f,g,\omega)\dvol(g):=
\partial_{\varepsilon}\left\{\left.P(x,e^{2\varepsilon f}g,\omega)\dvol(e^{2\varepsilon f}g)\right\}\right|_{\varepsilon=0}\,.$$
Since $P_1$ is linear in the jets of $f$, we have that
$$
P_1(f,g,\omega)=\sum_k\sum_{i_1,\dots,i_k}f_{;i_1\dots i_k}Q^{i_1\dots i_k}(g,\omega)\,.
$$
We integrate by parts formally to express
$$
P_1=\delta Q+\sum_k\sum_{i_1,\dots,i_k}(-1)^kfQ^{i_1\dots i_k}{}_{;i_k\dots i_1}\,.
$$
By assumption,
\begin{eqnarray*}
0&=&\partial_\varepsilon\left.\left\{\int_MP(x,e^{2\varepsilon f}g,\omega)\dvol
(e^{2\varepsilon f}g)\right\}\right|_{\varepsilon=0}=\int_MP_1(f,g,\omega)\dvol(g)\\
&=&\int_Mf\sum_k\sum_{i_1,\dots,i_k}(-1)^kQ^{i_1\dots i_k}{}_{;i_k\dots i_1}\dvol(g)\,.
\end{eqnarray*}
Since $f$ was arbitrary, 
$$\displaystyle\sum_k\sum_{i_1,\dots,i_k}(-1)^kQ^{i_1\dots i_k}{}_{;i_k\dots i_1}=0\,.$$
Thus $P_1(f,g,\omega)=\delta Q(f,g,\omega)$. We set $f=1$ to see 
\begin{equation}\label{E5.a}
P_1(1,g,\omega)=\delta Q(1,g,\omega)\,.
\end{equation}
We have that $\dvol(e^{2\varepsilon}g)=e^{m\varepsilon}\dvol(g)$.
Since $P$ is homogeneous of weight $n$, Lemma~\ref{L2.1} yields
$P(x,e^{2\varepsilon}g,\omega)\dvol(e^{2\varepsilon}g)=e^{(m-n)\varepsilon} P(x,g,\omega)\dvol(g)$.
Differentiating this identity shows $P_1=(m-n)P$; the Lemma now follows from Equation~(\ref{E5.a}).
\end{proof}

We now establish Assertion~(1c) of Theorem~\ref{T1.3}.
Let $P=r(\mathfrak{a}^{\operatorname{deR}}_{m,m})\in\JJ_{m-1,m}^0$.
By Lemma~\ref{L2.5}, $P$ is a multiple of $a_{m-1,m}^{\operatorname{deR}}$.
The hypothesis of Lemma~\ref{L5.1} are satisfied by Equation~(\ref{E1.a}). Consequently, Lemma~\ref{L5.1} shows that
$r(\mathfrak{a}^{\operatorname{deR}}_{m,m})=\delta Q_{m-1,m-1}^1$. By Lemma~\ref{L2.4}, we may choose $Q_{m,m-1}^1$ so
$r(Q_{m,m-1}^1)=Q_{m-1,m-1}^1$. Since $\delta r=r\delta$, $r(\mathfrak{a}^{\operatorname{deR}}_{m,m}-\delta Q_{m,m-1}^1)=0$. 
We use Lemma~\ref{L2.4} to see that
$$
\mathfrak{a}^{\operatorname{deR}}_{m,m}-\delta Q_{m,m-1}=c(m)\mathcal{E}_{m,m}\,.
$$
To evaluate the coefficient $c(m)$, we may take $\omega=0$ and use
(local) Poincare duality to see that $a_{m,m}(\Delta^p_{\mathcal{M}})=a_{m,m}(\Delta^{m-p}_{\mathcal{M}})$.
Consequently, we may show that $c(m)=\frac12m$ by computing:
\begin{eqnarray*}
&&2\mathfrak{a}^{\operatorname{deR}}_{m,m}=\sum_p(-1)^p p\cdot\left\{a_{m,m}(x,\Delta^p_{\mathcal{M}})+a_{m,m}(x,\Delta^{m-p}_{\mathcal{M}})\right\}\\
&=&\sum_p(-1)^pa_{m,m}(x,\Delta^p_{\mathcal{M}})\{p+m-p\}=m\cdot a_{m,m}^{\operatorname{deR}}\,.
\end{eqnarray*}

Suppose $m=2$. We use Lemma~\ref{L2.6} to see 
$-(4\pi)^{-1/2}a_{1,2}^{\operatorname{deR}}=(2\pi)^{-1}\delta\omega$.
Consequently, $r(Q_{2,1}^1)$ is non-zero.~\qed

\section{The proof of Theorem~\ref{T1.3}~(2c)}
There are several fundamental differences between the real and the complex
settings and we must treat the variables
$\{\omega_\alpha,\bar\omega_{\bar\beta}\}$ differently from the other variables.
In Section~\ref{S6.1}, we introduce some additional notation.
 Section~\ref{S6.2} is devoted to showing that in fact the heat trace invariants
$a_{m,n}(\Delta^{p,q})$ do not involve the $\{\omega_\alpha,\bar\omega_{\bar\beta}\}$
 variables but only the derivatives of positive order in $\omega$
are present; this uses a gauge renormalization.
 This is in marked contrast to the real setting. In Section~\ref{S6.3}, we complete the analysis;
again, there is a subtlety in that the conformal variation of a K\"ahler metric need no longer be K\"ahler.
Consequently, we pass (temporarily) to the Hermitian setting.
\subsection{Spaces of invariants}\label{S6.1}
We adopt the notation of Section~\ref{S2.2}. 
Recall that $\mathcal{U}\subset M_\mm(\mathbb{C})\otimes M_\ell(\mathbb{C})$ is the open subset consisting of all 
the matrices $(g,h)$
so that $g$ and $h$ define positive definite Hermitian inner products.
We
adjoin variables sucessively to the ground ring $C^\infty(\mathcal{U})$ to define the sub-rings:
\begin{eqnarray*}
&&\UU_m:=C^\infty(\mathcal{U})[g_{\alpha\bar\beta/U\bar V},h_{p\bar q/U\bar V}]_{|U|+|V|>0},\\
&&\VV_m:=\UU_m[\omega_{\alpha/UV},\bar\omega_{\bar\beta/UV}]_{|U|+|V|>0}\,.
\end{eqnarray*}
We explicitly exclude the variables $\{\omega_\alpha,\bar\omega_{\bar\beta}\}$ from $\UU_m$ and $\VV_m$.
Let $\UU_{m,n}\subset\UU_m$ and $\VV_{m,n}\subset\VV_m$ be the subspaces of weight $n$.
We impose no assumption of invariance as we are simply interested in the nature of the polynomials for the moment.
Let
$\End(\Lambda^{p,q}(M)\otimes E;\UU)$ and $\End(\Lambda^{p,q}(M)\otimes E;\VV)$ 
be the vector space of endomorphisms of
$\Lambda^{p,q}(M)\otimes E$ with coefficients in the rings $\UU$ and $\VV$, respectively.

\subsection{Normalizing the gauge}\label{S6.2}
The variables $\{\omega_\alpha,\bar\omega_{\bar\beta}\}$ are troublesome and we eliminate them from
consideration as follows. We work in the Hermitian setting as we have imposed no conditions on $\Omega$.

\begin{lemma}\label{L6.1}\
$a_{m,n}(x,\Delta^{p,q})\in\VV_{m,n}$, i.e., $a_{m,n}(x,\Delta^{p,q})$
does not involve the variables $\{\omega_\alpha,\bar\omega_{\bar\beta}\}$.
\end{lemma}

\begin{proof} We may decompose
$$
\bar\partial_{\bar\omega}=\operatorname{ext}(d\bar z^{\bar\beta})\left\{
\partial_{\bar\beta}+\bar\omega_{\bar\beta}\right\}\text{ and }
\delta^{\prime\prime}_\omega
=\operatorname{int}(dz^\alpha)\left\{-\partial_\alpha+\omega_\alpha\right\}+\mathcal{L}
$$
where $\mathcal{L}\in\End(\Lambda(M)\otimes E;\UU_{m,1})$ arises by taking the adjoint of $\bar\partial$
and is linear in the first derivatives of $g$ and of $h$;
$\mathcal{L}$ and $\operatorname{int}(dz^\alpha)$
lower the anti-holomorphic index $q$ by $1$ and $\operatorname{ext}(d\bar z^{\bar\beta})$ raises 
the anti-holomorphic index $q$ by $1$. We define:
\medbreak\qquad
$\mathcal{E}_1^{p,q}:=\operatorname{ext}(d\bar z^{\bar\beta})\operatorname{int}(dz^\alpha)\omega_{\alpha/\bar\beta}
\in\End(\Lambda^{p,q}(M)\otimes E;\VV_{m,2})$,
\smallbreak\qquad
$\mathcal{E}_2^{p,q}:=\operatorname{int}(dz^\alpha)\operatorname{ext}(d\bar z^{\bar\beta})\bar\omega_{\bar\beta/\alpha}
\in\End(\Lambda^{p,q}(M)\otimes E;\VV_{m,2})$,
\smallbreak\qquad
$\mathcal{L}^{p,q,\alpha}:=\operatorname{ext}(d\bar z^{\bar\beta})\partial_{\bar\beta}\{\operatorname{int}(dz^\alpha)\}\in
\End(\Lambda^{p,q}(M)\otimes E;\UU_{m,1})$,
\smallbreak\qquad
$\mathcal{L}^{p,q,\bar\beta}:=\operatorname{ext}(d\bar z^{\bar\beta})\mathcal{L}
+\mathcal{L}\operatorname{ext}(d\bar z^{\bar\beta})\in\End(\Lambda^{p,q}(M)\otimes E;\UU_{m,1})$,
\smallbreak\qquad
$\mathcal{Q}^{p,q}:=\operatorname{ext}(d\bar z^{\bar\beta})\partial_{\bar\beta}\{\mathcal{L}\}
\in\End(\Lambda^{p,q}(M)\otimes E;\UU_{m,2})$.
\medbreak\noindent
Since the matrices $\operatorname{ext}(dz^{\bar\beta})$ are constant with respect to the coordinate frame, we
do not need to introduce their derivatives. Note
that $\mathcal{E}^{p,q}_1$ and $\mathcal{E}^{p,q}_2$ are invariantly defined while $\mathcal{L}^{p,q,\alpha}$,
$\mathcal{L}^{p,q,\bar\beta}$, and $\mathcal{Q}^{p,q}$ are not invariantly defined but rather
depend on $(\vec z,\vec s)$. 
We use the identity
$$
\operatorname{int}(dz^\alpha)\operatorname{ext}(d\bar z^{\bar\beta})+
\operatorname{ext}(d\bar z^{\bar\beta})\operatorname{int}(dz^\alpha)=g^{\alpha\bar\beta}
$$
to express
$\Delta^{p,q}:=2\delta^{\prime\prime}_\omega\partial_{\bar\omega}+
2\partial_{\bar\omega}\delta^{\prime\prime}_\omega$
in the form:
\goodbreak\medbreak\qquad$\Delta^{p,q}=2g^{\alpha\bar\beta}\{-\partial_\alpha\partial_{\bar\beta}+
\omega_\alpha\partial_{\bar\beta}
-\omega_{\bar\beta}\partial_\alpha+\omega_\alpha\bar\omega_{\bar\beta}\}
+2\mathcal{E}^{p,q}_1-2\mathcal{E}^{p,q}_2+2\mathcal{Q}^{p,q}$
\smallbreak\hspace{1.8cm}$
+2\mathcal{L}^{p,q,\alpha}(-\partial_\alpha+\omega_\alpha)+2\mathcal{L}^{p,q,\bar\beta}(\partial_{\bar\beta}+\bar\omega_{\bar\beta})$.
\medbreak\noindent Fix a point $P$ of $M$ and let
$\xi:=\omega_\alpha(P)z^\alpha-\bar\omega_{\bar\beta}(P)z^{\bar\beta}$. Since $\xi$ is purely imaginary,
$\Delta^{p,q}_\xi:=e^{-\xi}\Delta^{p,q}e^\xi$ is defined by a unitary change of gauge. We compute
\medbreak\quad
$\Delta^{p,q}_\xi=2g^{\alpha\bar\beta}
\{-\partial_\alpha\partial_{\bar\beta}+(\omega_\alpha-\omega_\alpha(P))\partial_{\bar\beta}
-(\bar\omega_{\bar\beta}-\bar\omega_{\bar\beta}(P))\partial_\alpha\}$
\smallbreak\hspace{1.3cm}$+2g^{\alpha\bar\beta}
(\omega_\alpha-\omega_\alpha(P))(\bar\omega_{\bar\beta}-\bar\omega_{\bar\beta}(P))
+2\mathcal{E}^{p,q}_1-2\mathcal{E}^{p,q}_2+2\mathcal{Q}^{p,q}$
\smallbreak\hspace{1.3cm}$
+2\mathcal{L}^{p,q,\alpha}(-\partial_\alpha+\omega_\alpha-\omega_\alpha(P))
+2\mathcal{L}^{p,q,\bar\beta}(\partial_{\bar\beta}+\bar\omega_{\bar\beta}-\bar\omega_{\bar\beta}(P))$.
\medbreak\noindent Since $\Delta^{p,q}$ and $\Delta_\xi^{p,q}$ differ by a gauge transformation,
$a_{m,n}(\Delta^{p,q})=a_{m,n}(\Delta_\xi^{p,q})$. It is immediate that the symbol of $\Delta_\xi^{p,q}$
and all the derivatives of the symbol
of $\Delta_\xi^{p,q}$ at $P$ do not involve the $\omega_\alpha$ and $\bar\omega_{\bar\beta}$ variables. Consequently,
$a_{m,n}(\Delta_\xi^{p,q})\in\VV_{m,n}$.
\end{proof}

\begin{remark}\rm We have exploited a fundamental difference between the deformed real 
Laplacian and the deformed complex Laplacian
that relates to gauge freedom. We illustrate this as follows. Suppose $m=1$. We work on the circle. 
Let $\omega=adx$ for $a\in\mathbb{R}$.
Then $d_a=\partial_x+a$ and $\delta_a=-\partial_x+a$;
$\Delta_a^0=\Delta_a^1=\delta_ad_a=-(\partial_x^2+a^2)$ is already normalized optimally since there 
is no first order term and, unlike in the complex setting,
it is not possible to make a change of gauge that eliminates the dependence
of the symbol on $a$.
\end{remark}

\subsection{Conformal variations}\label{S6.3}

If $g$ is a K\"ahler metric,
then the conformal variation $e^{2\varepsilon}g$ need no longer be K\"ahler. However, this 
variation remains within the class of Hermitian manifolds.
Since Equation~(\ref{E1.a}) continues to hold for Hermitian manifolds, 
the argument given to prove Lemma~\ref{L5.1} shows that
we may express $a_{m-2,m}^{\operatorname{Dol}}=\delta Q_{m-2,m-1}^1$ in the class of Hermitian and
hence in the class of K\"ahler manifolds. Thus we may choose $Q_{m,m-1}^1$
so that
$$
\mathfrak{a}_{m,m}^{\operatorname{Dol}}-\delta Q_{m,m-1}^1\in\ker(r)\,.
$$
The integration by parts procedure discussed in Section~\ref{S5} arises from a conformal variation of $g$; thus
although additional derivatives of $\omega$ may be introduced, the number of
derivatives of $\omega$ are not reduced.
Thus $Q\in\VV_{m,m-1}^1$ and we have $\mathfrak{a}_{m,m}^{\operatorname{Dol}}-\delta Q\in\VV_{m,m}$.
Let
$$
\SS_m:=\mathbb{C}[\ch_k(TM,J,g),\ch_k(E,h),d\omega,d\bar\omega]\subset\TT_m
$$
be obtained by omitting the $\{\omega_\alpha,\bar\omega_{\bar\beta}\}$ variables. We restrict to the K\"ahler setting to obtain
$$
\mathfrak{a}_{m,m}^{\operatorname{Dol}}-\delta Q\in\star\{{\TT_m^m}\}\cap\VV_{m,m}=\star\SS_m^m\,.
$$
Since we have omitted the $\{\omega_\alpha,\bar\omega_{\bar\beta}\}$ variables, and since
$\ch_k$ is a closed differential form, we complete the proof of Theorem~\ref{T1.3} by computing
\begin{eqnarray*}
\SS_m^m&=&\CC_m^m+d\{\omega\wedge\SS_{m}^{m-2}+\bar\omega\wedge\SS_{m}^{m-2}\},\\
\star\SS_m^m&=&\star\CC_m^m+\delta\star\{\omega\wedge\SS_{m}^{m-2}+\bar\omega\wedge\SS_{m}^{m-2}\}\,.
\end{eqnarray*}

\section{The proof of Theorem~\ref{T1.4}}\label{S7} 
Let $\Delta^{p,q}_E$ be the complex Laplacian with coefficients in $E$.
Let $\Delta^n$ be the real Laplacian. If $P$ is a local invariant, let $P[\mathcal{M}]=\int_MP(\mathcal{M})\dvol$.
We will  use the following resuls in our analysis.
Assertion~1 is immediate from the definition, 
Assertions~(2,3) are Serre-duality, and Assertions~(4,5) follow
by specializing the Hirzebruch-Riemann-Roch Theorem.

\begin{theorem}\label{T7.1} Let $E$ be a holomorphic vector bundle over a K\"ahler manifold $\mathcal{M}$.
\ \begin{enumerate}
\item $\Delta^{p,q}_E=\Delta^{0,q}_{\Lambda^{p,0}\otimes E}$.
\item There is a conjugate linear isomorphism
intertwining $\Delta^{p,q}_E$ and $\Delta^{m-p,m-q}_{E^*}$.
\item Complex conjugation intertwines $\Delta^{p,q}$ and $\Delta^{q,p}$;
$\Delta^n=\oplus_{p+q=n}\Delta^{p,q}$.
\item If m=2, $\operatorname{index}(\bar\partial)=\{\frac12c_1(T_cM)\ch_0(E)+c_1(E)\}[M]$.
\item If m=4, $\operatorname{index}(\bar\partial)=\{\Td_2(T_cM)\ch_0(E)+\ch_1(T_cM)\ch_1(E)+\ch_2(E)\}[\mathcal{M}]$.
\end{enumerate}\end{theorem}

We use Leibnitz's formula to establish the following result.

\begin{lemma}\label{L7.2} Let $\mathcal{M}_j$ have real dimension $m_j$. Then
$$\aa_{m_1+m_2,m_1+m_2}[\mathcal{M}_1\times\mathcal{M}_2]=
\aa_{m_1,m_1}[\mathcal{M}_1]
 a_{m_2,m_2}^{\Dol}[\mathcal{M}_2]
+a_{m_1,m_1}^{\Dol}[\mathcal{M}_1]
\aa_{m_2,m_2}[\mathcal{M}_2]\,.$$
\end{lemma}

\begin{proof}Set $\AA_{m,n}[\mathcal{M}](s):=\sum_p(-1)^pa_{m,n}(\Delta^{(0,p)})[\mathcal{M}]s^p$.
We then have
$$
\left.a_{m,m}^{\Dol}[\mathcal{M}]=\AA_{m,m}[\mathcal{M}](s)\right|_{s=1}\text{ and }
\left.\aa_{m,m}[\mathcal{M}]=\partial_s\{\AA_{m,m}[\mathcal{M}]\}\right|_{s=1}\,.
$$
We have the following product formula for the heat trace asymptotics:
\begin{eqnarray*}
&&a_{m,n}(\Delta^{(0,p)})[\mathcal{M}_1\times\mathcal{M}_2]\\
&&\qquad=\sum_{p_1+p_2=p}\sum_{n_1+n_2=n}a_{m_1,n_1}(\Delta^{(0,p_1)})[\mathcal{M}_1]\cdot
a_{m_2,n_2}(\Delta^{(0,p_2)})[\mathcal{M}_2]\,.
\end{eqnarray*}
This then yields that
\begin{equation}\label{E7.a}
\AA_{m,n}[\mathcal{M}_1\times\mathcal{M}_2](s)=\sum_{n_1+n_2=n}\AA_{m_1,n_1}[\mathcal{M}_1](s)\cdot
\AA_{m_2,n_2}[\mathcal{M}_2](s)\,.
\end{equation}
\medbreak\noindent We differentiate Equation~(\ref{E7.a}) and set $s=1$  to obtain
$$
\aa_{m,m}[\mathcal{M}]=\sum_{n_1+n_2=m}
\aa_{m_1,n_1}[\mathcal{M}_1]\cdot a^{\Dol}_{m_2,n_2}[\mathcal{M}_2]+
a^{\Dol}_{m_1,n_1}[\mathcal{M}_1]\cdot\aa_{m_2,n_2}[\mathcal{M}_2]\,.
$$
We have that
$a^{\Dol}_{m_1,n_1}[\mathcal{M}_1]=0$ if $m_1\ne n_1$ and
$a^{\Dol}_{m_2,n_2}[\mathcal{M}_2]=0$ if $m_2\ne n_2$.
We may therefore safely set $n_1=m_1$ and $n_2=m_2$ in the above identity
to complete the proof of Lemma~\ref{L7.2}.
\end{proof}

The heat trace invariants are additive with respect to direct sums, i.e.
\begin{eqnarray*}
&&a_{m,n}(\Delta^{p,q})(M,g,J,E_1\oplus E_2,h_1\oplus h_2)\\
&=&
a_{m,n}(\Delta^{p,q})(M,g,J,E_1,h_1)+a_{m,n}(\Delta^{p,q})(M,g,J,E_2,h_2)\,.
\end{eqnarray*}
Since $a_{m,m}^{\Dol}$ and $\aa_{m,m}$
can be expressed in terms of characteristic classes, only the Chern character $\ch(E)$ enters. 
If $E$ is a line bundle, then $\ch_k(E,h)=\frac1{k!}c_1^k(E,h)$. 

\subsection{The proof of Theorem~\ref{T1.4}~(1a)}Suppose that $m=2$. 
There exist constants $\alpha^q$ and $\beta^q$ so that
$a_{2,2}(\Delta_E^{0,q})=\alpha^qc_1(T_cM)\ch_0(E)+\beta^qc_1(E)$.
We have $c_1(E^*)=-c_1(E)$. Since
$\Lambda^{1,0}(M)$ is the dual of $T_cM$, $c_1(\Lambda^{1,0}M)=-c_1(T_cM)$. We suppose $\dim E=1$ and
use Theorem~\ref{T7.1} to compute:
\begin{eqnarray*}
&&\{\alpha^1c_1(T_cM)+\beta^1c_1(E)\}[\mathcal{M}]=a_{2,2}(\Delta^{0,1}_E)[\mathcal{M}]\\
&&\qquad=a_{2,2}(\Delta^{1,0}_{E^*})[\mathcal{M}]=a_{2,2}(\Delta^{0,0}_{\Lambda^{1,0}(M)\otimes E^*})[\mathcal{M}]\\
&&\qquad=\{\alpha^0c_1(T_cM)+\beta^0c_1(\Lambda^{1,0}(M))+\beta^0c_1(E^*)\}[\mathcal{M}]\\
&&\qquad=\big\{(\alpha^0-\beta^0)c_1(T_cM)-\beta^0c_1(E)\}\big[\mathcal{M}]\,.
\end{eqnarray*}
Consequently, $\alpha^1=\alpha^0-\beta^0$ and $\beta^1=-\beta^0$. Therefore,
\begin{eqnarray*}
&&\operatorname{index}(\bar\partial_E)=\textstyle\frac12\{c_1(T_cM)+c_1(E)\}[\mathcal{M}]
=a_{2,2}^{\Dol}[\mathcal{M}]\\
&&\qquad=\{(\alpha^0-\alpha^1)c_1(T_cM)+(\beta^0-\beta^1)c_1(E)\}[\mathcal{M}]\\
&&\qquad=\{\beta^0c_1(T_cM)+2\beta^0c_1(E)\}[\mathcal{M}]\,.
\end{eqnarray*}
Consequently, $\beta^0=\frac12$ and $\beta^1=-\frac12$. 
Let $\mathcal{S}^2$ be the unit sphere in $\mathbb{R}^3$ with the usual metric and complex structure;
$\tau=2$
and $\operatorname{vol}(\mathcal{S}^2)=4\pi$. We compute
$$\textstyle\frac1{8\pi}\tau[\mathcal{S}^2]=1=\operatorname{index}(\bar\partial)[\mathcal{S}^2]=\frac12c_1[\mathcal{S}^2]$$
so $c_1(\mathcal{S}^2)=\frac1{4\pi}\tau$. Take $E$ trivial. McKean and Singer~\cite{MS67} (see also Patodi~\cite{P71}) 
computed $a_{2,2}(\Delta^p)$. Together with
Theorem~\ref{T7.1}, this shows
\begin{eqnarray*}
&&\textstyle a_{2,2}(\Delta^{0,0})=a_{2,2}(\Delta^0)=\frac1{24\pi}\tau=\frac16c_1(T_c),\\
&&\textstyle a_{2,2}(\Delta^{0,1})=\frac12 a_{2,2}(\Delta^1)=\frac12\frac1{24\pi}(-4\tau)=-\frac13c_1(T_c)\,.
\end{eqnarray*}
Thus $\alpha^0=\frac16$ and $\alpha^1=\alpha^0-\beta^0=\frac16-\frac12=-\frac13$.~\qed
\subsection{The proof of Theorem~\ref{T1.4} (1b)} Suppose $m=4$. Let $\mathcal{M}=(M,g,J,E,h)$ be a complex surface.
There exist constants $\alpha$ and $\beta$ and a characteristic class $P_2(T_cM)$ so that
$$
\aa_{4,4}[\mathcal{M}]=\{P_2(T_cM)\ch_0(E)+\alpha c_1(T_cM)c_1(E)+\beta\ch_2(E)\}[\mathcal{M}]\,.
$$
Let $\mathcal{M}=\mathcal{M}_1\times\mathcal{M}_2$. Suppose first that $\mathcal{M}_i$ are
flat tori and $(E_i,h_i)$ are line bundles with $c_1(E_i)[\mathcal{M}_i]\ne0$. 
Then
\begin{eqnarray*}
&&c_1(T_c(M))[\mathcal{M}]=0,\qquad \{P_2(T_cM)\ch_0(E)\}[\mathcal{M}]=0,\\
&&\ch_2(E_1\otimes E_2)[\mathcal{M}]=
c_1(E_1)[\mathcal{M}_1]\cdot c_1(E_2)[\mathcal{M}_2]\,.
\end{eqnarray*}
Thus Theorem~\ref{T7.1}, Lemma~\ref{L7.2} and Theorem~\ref{T1.4}~(1a) yield
\begin{eqnarray*}
&&\beta c_1(E_1)[\mathcal{M}_1]\cdot c_1(E_2)[\mathcal{M}_2]=\beta\ch_2(E_1\otimes E_2)[\mathcal{M}]=
\mathfrak{a}_{4,4}^{\Dol}[\mathcal{M}]\\
&&\qquad=\mathfrak{a}_{2,2}^{\Dol}[\mathcal{M}_1]
\cdot a_{2,2}^{\Dol}[\mathcal{M}_2]
+a_{2,2}^{\Dol}[\mathcal{M}_1]
\cdot\mathfrak{a}_{2,2}^{\Dol}[\mathcal{M}_2]\\
&&\qquad=\textstyle\frac12c_1(E_1)[\mathcal{M}_1]\cdot c_1(E_2)[\mathcal{M}_2]
+c_1(E_1)[\mathcal{M}_1]\cdot\frac12c_1(E_2)[\mathcal{M}_2]\\
&&\qquad=c_1(E_1)[\mathcal{M}_1]\cdot c_1(E_2)[\mathcal{M}_2]\,.
\end{eqnarray*}
This shows that $\beta=1$. Suppose $M_1=\mathcal{S}^2$, $E_1$ is trivial, $M_2$ is the flat torus,
and $c_1(E_2)[\mathcal{M}_2]\ne0$. Then
$
P_2(T_cM)\ch_0(E)=P_2(T_cM_1\oplus\pone)\ch_0(E)=0$ and $\ch_2(E)=\ch_2(E_2)=0\,.
$
Thus
\begin{eqnarray*}
&&\alpha c_1(T_cM_1)[\mathcal{M}_1]\cdot c_1(E_2)[\mathcal{M}_2]=\aa_{4,4}[\mathcal{M}]\\
&&\qquad=\aa_{2,2}[\mathcal{M}_1]
\cdot a_{2,2}^{\Dol}[\mathcal{M}_2]
+a_{2,2}^{\Dol}[\mathcal{M}_1]
\cdot\aa_{2,2}[\mathcal{M}_2]\\
&&\textstyle\qquad=\frac13c_1(T_cM_1)[\mathcal{M}_1]\cdot c_1(E_2)[\mathcal{M}_2]
+\frac12c_1(T_cM_1)[\mathcal{M}_1]\cdot\frac12c_1(E_2)[\mathcal{M}_2]\,.
\end{eqnarray*}
This shows that $\alpha=\frac13+\frac14=\frac7{12}$. Therefore, we have that
\begin{equation}\label{E7.b}
\textstyle\aa_{4,4}[\mathcal{M}]=\{P_2(T_cM)\ch_0(E)+\frac7{12}c_1(T_cM)c_1(E)+\ch_2(E)\}[\mathcal{M}]\,.
\end{equation}
We now use Serre duality. Let $E=\pone\oplus\Lambda^{2,0}(T_cM)$. We have
\begin{equation}\begin{array}{l}\label{E7.c}
\aa_{4,4}(\mathcal{M})=
-a_{4,4}(\Delta^{0,1}_{\pone\oplus\Lambda^{2,0}})+2a_{4,4}(\Delta^{0,2}_{\pone\oplus\Lambda^{2,0}})\\
\qquad=-a_{4,4}(\Delta^{2,1}_{\pone\oplus(\Lambda^{2,0})^*})
+2a_{4,4}(\Delta^{2,0}_{\pone\oplus(\Lambda^{2,0})^*})\\
\qquad=-a_{4,4}(\Delta^{0,1}_{\Lambda^{2,0}\oplus 1})+2a_{4,4}(\Delta^{0,0}_{\Lambda^{2,0}\oplus\pone})
\end{array}\end{equation}
We add the expressions of the second and third lines of Equation~(\ref{E7.c}) to see
\begin{equation}\label{E7.d}
2\aa_{4,4}(\mathcal{M})=2a_{4,4}(\Delta^{0,0}_E)-2a_{4,4}(\Delta^{0,1}_E)+2a_{4,4}(\Delta^{0,2}_E)
=2a_{4,4}^{\Dol}(\mathcal{M})\,.
\end{equation}
We use Equation~(\ref{E7.b}), Equation~(\ref{E7.d}), and Theorem~\ref{T1.4} to see
\begin{eqnarray*}
&&\textstyle\{P_2(T_cM)\ch_0(E)+\frac7{12}c_1(T_cM)c_1(E)+\ch_2(E)\}[\mathcal{M}]\\
&=&\textstyle\{\Td_2(T_cM)\ch_0(E)+\frac12c_1(T_cM)c_1(E)+\ch_2(E)\}[\mathcal{M}]\,.
\end{eqnarray*}
We have $\ch_0(E)=2$ and $\ch_1(E)=c_1(\Lambda^{2,0})=-c_1(T_cM)$. Consequently
\medbreak
$P_2(T_cM)=\Td_2(T_cM)+\frac12(\frac12-\frac7{12})(-c_1^2(T_cM))=\Td_2(T_cM)+\frac1{24}c_1^2(T_cM)$.~\qed

\subsection{The proof of Theorem~\ref{T1.4}~(2a,2b)}
Let $\mathcal{M}:=\mathcal{N}_1\times\dots\times\mathcal{N}_\mm$ where
$(N_i,g_i,J_i)$ are flat tori of real dimension 2 and $(E_i,h_i)$ are Hermitian line bundles over $N_i$.
Let $E=E_1\otimes\dots\otimes E_{\mm}$. Then
$\ch_\mm(E)=c_1(E_1)\dots c_1(E_\mm)$. We apply Lemma~\ref{L7.2} and Theorem~\ref{T1.4}
to prove Theorem~\ref{T1.4}~(2a) by computing
\medbreak\qquad
$a_{2,2}^{\Dol}[\mathcal{M}]=\prod_ic_1(E_i)[\mathcal{N}_i]$,
\qquad
$\textstyle\AA_{2,2}(s)[\mathcal{N}_i]=\frac12(1+s)c_1(E_i)[\mathcal{N}_i]$,
\medbreak\qquad
$\aa_{m,m}[\mathcal{M}]=\partial_s\left\{\AA_{m,m}(s)[\mathcal{M}]\right\}|_{s=1}=\left.\partial_s\left\{2^{-\mm}(1+s)^\mm\prod_i c_1(E_i)[\mathcal{N}_i]\right\}\right|_{s=1}$
\medbreak\qquad$\hspace{1.5cm}
=\frac12\mm\ch_k(E)[\mathcal{M}]=\frac12\mm _{2,2}^{\Dol}[\mathcal{M}]$. 
\medbreak\noindent
Similarly, let $\mathcal{M}_i$ be arbitrary Riemann surfaces and let $(E_i,h_i)$ be trivial.
We complete the proof of Theorem~\ref{T1.4}~(2b) by computing
\medbreak\qquad
$a_{2,2}^{\Dol}[\mathcal{M}]=2^{-\mm}\prod_ic_1(T_c)[\mathcal{M}_i]$,
\medbreak\qquad
$\aa_{m,m}[\mathcal{M}]=\partial_s\left\{\AA_{m,m}(s)[\mathcal{M}]\right\}|_{s=1}=\partial_s\left\{6^{-\mm}(1+2s)^\mm\prod_i c_1(T_c)[\mathcal{M}_i]\right\}|_{s=1}$
\medbreak\qquad
$\phantom{\aa[\mathcal{M}]}
=\frac13\mm2^{-\mm+1}\prod_i c_1(T_c)[\mathcal{M}_i]=\frac23\mm a_{m,m}^{\Dol}[\mathcal{M}]$.~\qed

\subsection*{Research Support} Research partially supported by grants PID2019-105138GB-C21 (Spain) and MTM2017-89686-P (AEI/FEDER, UE).
\subsection*{This paper is dedicated to the memory of Professor L. Nirenberg}\ 
Gilkey writes: 
Profesor L. Nirenberg  was my thesis advisor and mentor.
My thesis dealt with a heat equation proof of the index theorem for geometrical elliptic complexes; one of
the chapters dealt with the Dolbeault complex in the K\"ahler setting. Much has happened since I received
my Ph.~D.\ in 1972 from Harvard and I have revisited the topic several times subsequently. But it seems appropriate
to revisit the topic for what I expect will be one final time with my colleague Professor \'Alvarez L\'opez who is
an expert on the Witten deformation and a valued collaborator and friend.
The years 1969--1972 when I was on ``travelling guidance" from Harvard studying with Professor Nirenberg at NYU
were tumultuous ones as that was the Vietnam war era with an occupation of the Courant Institute and many anti-war
protests. Professor Nirenberg 
helped me obtain student deferments that enabled me to complete
my degree and on those grounds alone I owe him a debt 
of gratitude I can never repay. But in addition,
he offered me wise professional and personal
guidance and in all manner enriched my life and helped me launch my mathematical career.
It is an honor to dedicate this paper to his memory as I near the end of my
own career in mathematics. He launched me on a wonderful journey through life.

\'Alvarez L\'opez writes: It is a great honor for me to participate in a scientific contribution to the memory of Louis Nirenberg. The first time I heard his name was in a PhD course, around 1985, studying the celebrated Newlander-Nirenberg integrability condition for the Dolbeault complex. After that I continued to discover the enormous importance and variety of his achievements, which made him one of the most relevant mathematicians of the 20th century. It was a very pleasant surprise when Peter Gilkey told me that Louis Nirenberg had been his thesis advisor. I was fortunate to collaborate with Peter Gilkey, and to have that indirect connection with Louis Nirenberg.

\end{document}